\RequirePackage{fix-cm}
\documentclass[a4paper,12pt]{article}
\usepackage{arxiv}

\usepackage{graphicx}
\usepackage{xcolor}
\definecolor{abstract_background}{RGB}{235,235,235}%
\definecolor{ao(english)}{rgb}{0.0, 0.5, 0.0}
\usepackage{hyperref}
\hypersetup{colorlinks,linkcolor={black},citecolor={ao(english)},urlcolor={black}}
\usepackage{mathtools}
\usepackage{bbm}
\usepackage{subfigure}
\usepackage{placeins}
\usepackage{babel}
\usepackage{bm}
\usepackage{amstext,amsmath,amssymb,amsbsy,bm,amsthm}
\usepackage{verbatim}
\theoremstyle{definition}

\newtheorem{theorem}{Theorem}[section]
\newtheorem{proposition}{Proposition}[section]
\newtheorem{remark}{Remark}

\newtheorem{example}[theorem]{Example}

\def\R{{\mathbb R}}
\def\N{{\mathbb N}}

\newcommand{\shorttitle}{Linear Stability Analysis of PI-RPNNs for ODEs}
\title{Linear Stability Analysis of Physics-Informed Random Projection Neural Networks for ODEs}

\author{
\textbf{Gianluca Fabiani\textcolor{teal}{$^{1}$}, Erik Bollt\textcolor{teal}{$^{2}$},
Constantinos Siettos\textcolor{teal}{$^{3,}$}\thanks{Corresponding authors, emails: \texttt{constantinos.siettos@unina.it, ayannaco@aueb.gr}} , Athanasios N.  Yannacopoulos$^{\textcolor{teal}{4},*}$} \\
{}\\
\textcolor{teal}{$^{(1)}$} Modelling Engineering Risk and Complexity, \emph{Scuola Superiore Meridionale}, Naples 80138, Italy \hspace{1cm}\\
\textcolor{teal}{$^{(2)}$}Dept. of Electrical and Computer Engineering, Clarkson University, Potsdam, NY, USA\\
\textcolor{teal}{$^{(3)}$}Dipartimento di Matematica e Applicazioni ‘‘Renato Caccioppoli", \emph{Universit\`a degli Studi di Napoli}\\ \hspace{0.38cm}\emph{Federico II}, Naples 80126, Italy\\
\textcolor{teal}{$^{(4)}$}Department of Statistics, \emph{Athens University of Economics and Business}, Greece\\
}

\begin{document}
\maketitle

\begin{abstract}
We present a linear stability analysis of physics-informed random projection neural networks (PI-RPNNs), for the numerical solution of {the initial value problem (IVP)} of (stiff) ODEs. We begin by proving that PI-RPNNs are uniform approximators of the solution to ODEs. We then provide a constructive proof demonstrating that PI-RPNNs offer consistent and asymptotically stable numerical schemes, thus convergent schemes. In particular, we prove that multi-collocation PI-RPNNs guarantee asymptotic stability. Our theoretical results are illustrated via numerical solutions of benchmark examples including indicative comparisons with the backward Euler method, the midpoint method, the trapezoidal rule, the 2-stage Gauss scheme, and the 2- and 3-stage Radau schemes. 
\end{abstract}



\section{Introduction}
Machine learning (ML) as a tool for the solution of differential equations, can be traced back to the '90s \cite{lee1990neural,gerstberger1997feedforward,lagaris1998artificial}. More recently, theoretical and technological advances have bloomed research activity in the field, with  Physics-Informed (Deep) Neural Networks (PINNs) \cite{raissi2019physics,lu2021deepxde,ji2021stiff} being a major development.  Other ML schemes, include Gaussian Processes (GPs), \cite{raissi2018numerical, chen2021solving, chen2024sparse}, 
generative adversarial networks (GANs)\cite{yang2020physics}, neural operators \cite{kovachki2023neural, cao2024laplace}, and the random features model (RF) for learning the solution operator of PDEs \cite{nelsen2021random}. \color{black}
However, such ML schemes often provide moderate numerical accuracy, with much  higher computational costs when compared to traditional methods. Besides, such schemes suffer from the curse of dimensionality, with the usual problems of possible trapping in suboptimal regimes (see also discussions in \cite{wang2021understanding, de2022cost, goswami2023physics}).
An open and challenging problem, revolves around the design and theoretical analysis
of resource-limited, machine learning schemes that can challenge other well-
established numerical analysis techniques.\par
A promising approach to this challenge are various versions of Random Projection Neural Networks (RPNNs) (also called Random Feature Neural Networks,  Extreme Learning Machines \cite{huang2006extreme} and Reservoir computing \cite{jaeger2002adaptive}), i.e.,  single layer feedforward neural networks where the basis functions, internal weights and biases are  randomly  chosen and kept fixed, with the only unknowns being the coefficients in the linear expansion (for a review see in \cite{scardapane2017randomness, galaris2022numerical, fabiani2023parsimonious}. RPNNs have been used for the solution of both the inverse \cite{san2018extreme, galaris2022numerical, fabiani2024task, dong2023method, gonon2023random, fabiani2025randonet, fabiani2025enabling} and the forward problem \cite{yang2018novel, sun2019solving, schiassi2021extreme, fabiani2021numerical, dong2021modified, dong2021local, calabro2021extreme, de2022physics, fabiani2023parsimonious, patsatzis2024slow, datar2024solving} in differential equations. Conceptually, RPNNs have been introduced by Rosenblatt \cite{rosenblatt1962perceptions} in the 1960s, and further implemented and theorized in the 1990s \cite{barron1993universal,igelnik1995stochastic}. A key concept (among others) is the celebrated Johnson-Lindenstrauss Lemma \cite{johnson1984extensions}, which ensures accurate embeddings in a lower (or even a higher dimensional) space that can be implemented via random projections. A prominent example of such a scheme is the Random Vector Functional-Link Network (RVFL) introduced by Igelnik and Pao  (see Thm.  3 in \cite{igelnik1995stochastic}) where the set of random features model   
$f_{N}(x ; \omega)=\sum_{j=1}^N w_j \psi (a_j \cdot x +b_j)$  for activation functions $\psi \in L^{2}(\R)$ or $H^{1}(\R)$  and i.i.d random  parameters $a_j$, $b_j$  is dense (i.e., can approximate arbitrarily well)  real continuous functions on compact  $K \subset \R^d$,  with a rate of convergence $O(N^{-1/2})$ as well as the pioneering work of Rahimi and Recht
(see Thms 3.1 and 3.2 in \cite{rahimi2008uniform})) with similar expansions in appropriate RKHS. Recently, based on these previous works, we introduced and proved that \textit{RandONets} -- a neural operator architecture based on RPNNs -- are universal approximators of functionals and nonlinear operators \cite{fabiani2025randonet}.
\par 
At this point, we highlight the following issues when comparing random projection-based methods and similar traditional/deterministic ones, such as radial basis function (RBF) expansions trained with greedy approaches. First, both greedy approximations with RBF expansions and RPNNs exhibit a convergence rate of order $\mathcal{O}(\frac{1}{ \sqrt{N}})$ independently of the dimensionality \cite{gorban2016approximation}. However, random projection-based techniques offer greater scalability in addressing the ``curse of dimensionality". Unlike greedy approaches, which require the solution of a nonlinear optimization problem at each step, randomized approaches simplify the training process by drawing kernel parameters randomly, reducing the optimization problem into a linear one. On the other hand, random projection-based methods require a careful design of the sampling strategy of the hyperparameters and activation functions for achieving the so-called multi-scale effect in order to improve the generalization of the model (see also the critical discussion in \cite{gorban2016approximation}). Such issues related to the convergence properties and quality of the (probabilistic) approximation, but also the stability properties, of RPNNs, remain open and challenging problems.\par
\color{black}
Our work motivated by the above challenges presents a rigorous \textit{constructive} proof of the stability properties of PI-RPNNs for IVPs of ODEs. We emphasize that, our aim is not to extensively compare this scheme with other traditional schemes or other machine learning approaches. Such detailed comparisons, have been provided in our previous works (see \cite{fabiani2021numerical, fabiani2023parsimonious, fabiani2025randonet}), where optimally designed PI-RPNNs are shown to rival FD, FEMs, and adaptive stiff solvers for ODEs while outperforming physics-informed deep neural networks by orders of magnitude in accuracy and efficiency. 
To the best of our knowledge, there is no other study providing a systematic theoretical constructive stability analysis of PI-RPNNs for the solution of the solution of differential equations; in other works the stability properties and performance of ML schemes, in the context of PINNs, are usually demonstrated only via numerical experiments.
%
The structure of the paper is as follows. In section \ref{sec:PINN_ODEs}, we introduce an appropriate representation of PI-RPNNs for the solution of the IVP problem of ODEs and demonstrate that PI-RPNNs are uniform approximators for such problems. In section \ref{sec:stability}, we provide a constructive proof for the design of PI-RPNNs that guarantees asymptotic stability for the solution of linear ODEs. We furthermore prove that PI-RPNNs are consistent, thus convergent schemes. Our theoretical results are illustrated in section \ref{sec:num_results}, via numerical solutions of two benchmark problems including a system of stiff ODEs, and a stiff linear parabolic PDE, discretized with finite differences. We also present a comparative analysis of the PI-RPNNs versus standard implicit schemes, assessing both numerical accuracy and computational efficiency across different step sizes.
\section{Random Projection neural networks (RPNNs) as approximators to the solution of ODEs}
\label{sec:PINN_ODEs}
\color{black}
Here we consider the solution  $\bm{u}=(u^{(1)},\cdots,u^{(d)})$ of the initial value problem (IVP) of ODEs, with $d$ variables, of the form:
\vspace{-1mm}
\begin{equation}
    \frac{d\bm{u}(t)}{dt}=\bm{f}\big(t,\bm{u}(t)\big), \qquad t\in [t_0, T], \qquad \bm{u}(t_0)=\bm{u}_0,
    \label{eq:nonlinODE}
\end{equation}
where $\bm{f}:\R\times \R^d\rightarrow \R^d$ satisfies the assumptions of the Picard-Lindel\"of theorem (in particular continuous in $t$ and  Lipschitz in $u$), so that a unique solution $u \in C(K ; {\mathbb R}^{d})$ (i.e. continuous) exists for \eqref{sec:PINN_ODEs}, for any $K \subset [t_0, T]$, compact. Then, by the change of variables $t \mapsto \tau$ with  $t= t_0 + \tau (T- t_0)$ we obtain the integral representation
\vspace{-1mm}
\begin{equation}
\bm{u}(t)=\bm{u}_0+(t-t_0)\int_{0}^{1}\bm{f}(\tau(t-t_0)+t_0,\bm{u}(\tau(t-t_0)+t_0))d\tau,
\label{eq:PLInt-0}
\end{equation}
\color{black}
which is approximated (component-wise) in terms of the RPNN 
(see \cite{fabiani2023parsimonious}):
\begin{equation}
\hat{u}^{(k)}(t,\bm{w}_k)=u_{0}^{(k)}+(t-t_0) \sum_{j=1}^N w_{j,k} \phi(t,\tau_j,\theta_j), \quad \phi(t,\tau_j,\theta_j)=\exp{(-\alpha_U \theta_j (t-\tau_j)^2)},
\label{eq:trsol-0}
\end{equation}
where  
the random shape parameters $\bm{\theta}=(\theta_1,\theta_2,\cdots,\theta_N)$ are i.i.d. draws from a suitable distribution (here $\theta_{i} \sim {\cal U}[0, 1]$), $\alpha_{U}$ is a deterministic parameter to be specified so as to ensure stability); the centers $\bm{\tau}=(\tau_1,\tau_2,\dots,\tau_N)$ equidistant and deterministic, and $w_{j,k}$ are the $N$ unknown  random weights to be obtained by minimizing an appropriate error functional. Note that \eqref{eq:trsol-0} satisfies explicitly the initial condition at $t=t_0$.
\par
Physics Informed RPNNs (PI-RPNNs) are approximators for the solutions of \eqref{eq:trsol-0}, where the error functional involves the form of the ODE at given collocation points.
For $M$ collocation points, $c_i$, $i=1, \cdots, M$, the training of $w_{j,k}$ is carried out  by minimizing the $M$ nonlinear residuals ${\bf F}_i \in \R^{d}$
\vspace{-1mm}
\begin{equation}
{\bf F}_i({\bf w})= \dfrac{d\hat{{\bf u}}}{dt} (c_i,\bm{w}) -{\bf f}(c_i, \hat{{\bf u}}(c_i,\bm{w})), \,\,\,\, i=1, \cdots, M.
\label{eq:Fq}
\end{equation}
This is a vector minimization problem that can be scalarized in terms of the minimization of the functional
\begin{eqnarray}
{\cal L}_{\delta}({\bf w})= {\mathbb E}\left[ \frac{1}{2} \sum_{i=1}^{M} \| {\bf F}_{i} \|_{{\mathbb R}^d}^2  + \frac{\delta}{2} \| {\bf w} \|_{{\mathbb R}^{N \times d}}^2\right],
\label{OPT-24-2-205}
\end{eqnarray}
where $\delta >0$ is a regularization parameter and by $\| {\bf w} \|_{{\mathbb R}^{N \times d}}$ we denote the Frobenius norm of the matrix ${\bf w}=(w_{j,k})$.
 The solution of the above non-linear least-squares problem can be obtained, e.g. with quasi-Newton, Gauss-Newton,
but also with truncated SVD decomposition, or, QR factorization with regularization \cite{galaris2022numerical, fabiani2023parsimonious, fabiani2025random}, included to ensure better properties of the estimators of the solution in terms of variance (as for example in ridge regression).\par

At this point, we will demonstrate the following proposition:

\begin{proposition} 
RPNNs  of the form \eqref{eq:trsol-0}
 are universal approximators of the solution to set of ODEs  \eqref{eq:nonlinODE} in $C(K ; \R^d)$, where $K \subset [t_0, T]$, compact.  
\end{proposition}

\begin{proof}(Sketch) It suffices to prove the result for the case $d=1$, and we simplify the notation by dropping the super/sub scripts $k$ refering to the components of ${\bf u}$ in the expansion \eqref{eq:trsol-0}.


We start with the equivalent representation \eqref{eq:PLInt-0} with the IVP \eqref{eq:nonlinODE},
and define the continuous function 
\begin{eqnarray*}
t \mapsto :\Psi(t) = \int_{0}^{1}\bm{f}(\tau(t-t_0)+t_0,\bm{u}(\tau(t-t_0)+t_0))d\tau.
\end{eqnarray*}
The continuity of $\Psi$  on $K \subset \R_{+}$ follows by an application of the Lebesgue dominated convergence theorem (using 
the Lipschitz  continuity assumption on $u \mapsto f(t, u)$ and an extra boundedness or integrability assumption on $t \mapsto f(t, u)$ ).

Consider the family of feature functions ${\cal G}=\{ g_i \,\, : \,\, i=1, \cdots, N\}$, where $t \mapsto g_i(t ; \theta):=\exp(-\alpha_{U} \theta (t -\tau_i)^2)$, with $\theta \in \Theta$ a scale parameter (to be randomized) and $\tau_i$ the deterministic location parameters (centers) of the Gaussians. For appropriate choice of the centers $\{\tau_i\}$, the family $G(\Theta) := span( \{ g(\cdot ; \theta) \,\, : \,\, g \in {\cal G}, \,\, \theta \in \Theta\})$ is dense in the function space $X$ (chosen to be either $C(K ; \R)$ or $L^{p}(K ; \R)$ for an appropriate $p \in [1, \infty]$).
Note that $G(\Theta)$ contains finite sums of the form $\sum_{i=1}^{N} c_i g_{i}(t ; \theta_i)=\sum_{i=1}^{N} c_i \exp(- \alpha_{U}\theta_i (t - t_i)^2)$ for some $\theta_i \in \Theta$, hence the stated density results follows by standard arguments related to 
the universal approximation theorem for RBF (see, e.g., \cite{park1993approximation} Thm 2, or \cite{liao2003relaxed}  Thm. 1) applied for  Gaussian RBF.
This step can be used to approximate the continuous function $\Psi$ in terms of an expansion in suitably selected Gaussian RBF, which leads to \eqref{eq:trsol-0}. This settles the validity of the expansion \eqref{eq:trsol-0} for deterministic choice of the parameters. Note that by the continuity of the Gaussians, the compactness of $K$ and by assuming a sufficiently dense set of (preassigned) collocation points  $\tau = (\tau_1, \cdots, \tau_N)$ we can approximate the required function $\Psi$ by a deterministic sum  of the form 
 \begin{eqnarray}\label{15-2-2025}
\Psi(t)=\sum_{j=1}^N c_j g_{j}(t; \theta_j)= \sum_{j=1}^{N} c_j \phi(t,\tau_j,\theta_j),
\end{eqnarray}
where $\phi(t,\tau_j,\theta_j)=\exp{(-\alpha_U \theta_j (t-\tau_j)^2)}$.
This leads to  the deterministic form of our proposed expansion  (see \eqref{eq:trsol-0} and comment in the beginning of the proof).

We now consider the problem of the actual random feature representation, proposed in  \eqref{eq:trsol-0}. This requires exchanging the deterministic parameters $\theta_j \in \Theta$ in \eqref{15-2-2025} by an i.i.d. random sample $\theta_1, \cdots, \theta_N \sim P$ (from a probability distribution $P$ satisfying a full support condition -- here chosen to be the uniform distribution on a suitable interval), while keeping the $\tau_j$ deterministic (and subsequently   also randomizing the weights $c_j$).
 Our random features (projection) neural network then assumes the form 
 \begin{equation}\label{15-2-2025-1}
 G_{N}(t, \omega)= \sum_{i=1}^{N} y_{n}(\omega) g_{i}(t; \theta_{i}(\omega)),
 \end{equation}
 for $\theta_i$, $i=1, \cdots, N$, i.i.d. random variables drawn by a probability distribution satisfying the full support assumption (see Assumption 3.1 
in \cite{neufeld2023universal}), and for the choice $g_i(\cdot, \theta)=\phi(\cdot , \tau_i, \theta)$ (with $\phi$ as above) we obtain an expansion for $\Psi$ which upon substitution renders \eqref{eq:trsol-0}. The set of random functions of the form \eqref{15-2-2025-1} will be denoted by ${\cal RG}$. To conclude the proof and show the universal approximation property of expansion of the form \eqref{eq:trsol-0} to the function space $X$ (chosen here $X=C(K ; \R)=C(K)$) it suffices to prove the density of ${\cal RG}$ in $L^{r}(\Omega, {\cal F}_{\Theta}, P ; X)$ for $r \in [1, \infty)$, where $(\Omega, {\cal F}_{\Theta}, P)$ is the probability space related to the randomization of the parameters $\alpha$. This step is provided 
 by applying Theorem 3.2 in \cite{neufeld2023universal}.
\end{proof}


\begin{remark}
Concerning the rate of convergence of the expansion in \eqref{eq:trsol-0} one may employ results in the spirit of Thm. 4.5 in \cite{neufeld2023universal}. These strongly depend on the nature of the function to be approximated (in terms of the type of the Banach space $X$ (containing the functions to be approximated)  and the probabilistic Bochner space in which the approximation is required. For example, if the solutions is such that the function $\Psi$ is in $L^{2}$ (or an appropriate Sobolev space) then the rate of convergence of the expansion is of the order $O(N^{-1/2})$, i.e., of the same order of magnitude as for the Rahimi and Recht result. Importantly, the solutions the ODEs considered here enjoy membership in such Sobolev spaces.
\end{remark}

In the following section, we will convey the stability analysis for both multi-collocation and single-collocation ($M=1$) schemes. We note that the analysis can be naturally extended to nonlinear systems of ODEs near their steady states. 

\color{black}
\section{Linear Stability Analysis of {PI-RPNNs}}
\label{sec:stability}
For studying the linear stability of RPNNs we must consider the linearization of problem \eqref{eq:nonlinODE} around an equilibrium, which assumes the form 
\vspace{-1mm}
\begin{equation}
    \frac{d\bm{u}(t)}{dt}=A \bm{u},  \qquad \bm{u}(t_0)=\bm{u}_0,
    \label{eq:simplelinear0}
\end{equation}
where
$A \in \R^{d \times d}$ is the corresponding Jacobian matrix of ${\bf f}$ at the chosen equilibrium, and $\bm{u}_0 \in \R^{d}$ is a given initial condition.
Breaking up $[t_0, T]$ into a sequence of sub-intervals $[t_{\ell}, t_{\ell + 1}]$, $\ell \in \N$, and using RPNNs of the form \eqref{eq:trsol-0} for each such interval we produce 
 a sequence $\{ \hat{{\bf u}}_\ell \}_{\ell \in {\mathbb N}} \subset \R^{d}$ of random vectors, where, for each $\ell \in {\mathbb N}$, $\hat{u}_\ell$ is an approximation of the solution of the system \eqref{eq:trsol-0} in the interval $[t_{\ell},t_{\ell+1})$. Hence, the RPNN scheme may be expressed as:
\vspace{-1mm}
\begin{eqnarray}\label{TH-COMPACT}
\hat{\bm{u}}_{\ell+1}(\omega) = S(\theta, \omega) \hat{\bm{u}}_{\ell}(\omega), 
\vspace{-1mm}
\end{eqnarray}
where we use $\omega$ to emphasize that the output of the proposed numerical method is random on account of the randomization of the parameters $\theta \in \Theta$ (in our case $\theta=\alpha$) and $S(\theta, \omega)$ is a random matrix. As an intermediate, important, step in our analysis, we
 first consider the stability analysis of the linear scalar ODE.
Then, in Section \ref{sec:system_stability} we will use these results to extend our analysis to the system of ODEs \eqref{eq:simplelinear0}.
\subsection{The case of linear scalar ODEs}
We first consider the scalar problem (d=1), setting $A=\lambda \in \R$ in \eqref{eq:simplelinear0}, i.e.,
\begin{equation}\label{graf-2025}
\frac{du(t)}{dt}=\lambda u, \quad u(t_0)=u_0,
\end{equation}
so that $S(\theta, \omega) \in \R$.  The system is asymptotically stable if $|S| < 1$ and unstable if $|S| > 1$.

Next, we prove  that the PI-RPNN scheme is asymptotically stable (even for very large stiffness). Then, we prove the consistency of the scheme.
For  analysis, we choose fixed equidistant collocation points $c_{i} \in (t_{\ell-1}, t_{\ell}]$, $i =1, \cdots, M$, defined as:
\vspace{-1mm}
\begin{equation}\label{kalamata-collocation-1}
    c_{{i}}=t_{\ell-1}+(t_{\ell}-t_{\ell-1})\zeta_i=t_{\ell-1}+h\zeta_{i}, \quad \zeta_{i}=\frac{i}{M}.
    \vspace{-1mm}
\end{equation}
We also define for each $\ell$, the matrix  $\Psi^{(\ell)} \in \R^{N \times M}$, with elements
\begin{equation}
\psi_{j i}^{(\ell)}:=\psi_{j}(c_{i},\theta^{(\ell)}_j)
:=\phi(t,\tau^{(\ell)}_j,\theta^{(\ell)}_j) + (c_i-t_{\ell-1}) \phi'(c_i,\tau^{(\ell)}_j,\theta^{(\ell)}_j) - \lambda  (c_i - t_{\ell-1}) \phi(c_i,\tau^{(\ell)}_j,\theta^{(\ell)}_j),
\end{equation}
where $j=1, \cdots, N$ and $i=1, \cdots, M$, as well as the vector
\begin{equation}
\Phi^{(\ell)}(t) = (\phi(t, \tau_{1}^{(\ell)},\theta_{1}^{(\ell)}), \cdots, \phi(t, \tau_{N}^{(\ell)},\theta_{N}^{(\ell)}))^{T} = (\phi^{(\ell)}_1(t), \cdots, \phi^{(\ell)}_N(t) )^{T} \in {\mathbb R}^{N \times 1}.
\label{matrix-form-2025}
\end{equation}
We will also use the notation $\hat{u}_{\ell}:=\hat{u}(t_{\ell})$, which does not necessarily (as an approximation) coincide with $u(t_{\ell})$, except when $\ell=0$, i.e., at the initial condition.

The training of the RPNN, {within the framework of PINNs, called hereafter a PI-RPNN} is done by choosing the weights as minimizers of the (family) of loss functionals
\vspace{-1mm}
\begin{eqnarray}\label{LOSS-2025}
{\cal L}_{\delta}^{(\ell)}(w)= {\mathbb E}[{\cal L}^{(\ell)}(w ; \omega)] = {\mathbb E} \bigg[ \frac{1}{2} \sum_{i=1}^{M} (\hat{u}'(c_i ; w) -\lambda \hat{u}(c_i ; w))^2 + \frac{\delta}{2} \sum_{j=1}^{N} | w_{j} |^2 \bigg], 
\end{eqnarray}
$\delta > 0$, for $c_i \in (t_{\ell-1}, t_{\ell}]$ as in \eqref{kalamata-collocation-1} above. The second term in ${\cal L}^{(\ell)}_{\delta}$ is a regularization term, which allows for a unique solution to the minimization problem. 
\begin{proposition}\label{KALAMATA-PROPOSITION} 
The regularized multi-collocation scheme for \eqref{graf-2025}, provides $\forall t \in (t_{\ell -1}, t_{\ell})$ the solution 
for each of the collocation points $c_i \in (t_{\ell -1}, t_{\ell})$ (defined as in \eqref{kalamata-collocation-1}):
\begin{equation}\label{peroni-1000}
\begin{aligned}
\hat{u}_{\ell}(c_i)=\hat{u}_{\ell -1} \bigg( 1 +  \lambda {h} \zeta_i (\Phi^{(\ell)}_{i})^{T}  (\Psi^{(\ell)} (\Psi^{(\ell)} )^{T}  + \delta  I)^{-1} \Psi^{(\ell)}    {\bf 1}_{M \times 1}  \bigg), i=1, \cdots, M, \,\,\, \delta > 0.
\end{aligned}
\end{equation}
In the limit as $\delta \to 0^{+}$ this reduces to
\vspace{-1mm}
\begin{eqnarray*}
u(t)=\hat{u}_{\ell -1} + (t -t_{\ell-1}) (\Phi^{(\ell)})^{T}(t) w,
\end{eqnarray*}
where $w$ is 
the least-squares solution of the collocation system
\begin{equation}
\sum_{j=1}^{N} w_{j}^{(\ell)} \psi_{ji}^{(\ell)}=\lambda  u(t_{\ell-1}), \,\,\, i=1, \cdots, M \,\,\, \Longleftrightarrow \,\,\,
     \, (\Psi^{(\ell)}_{N\times M} )^{T} \,\bm{w}^{(\ell)}_{N\times 1}  =\lambda u(t_{\ell-1}){\bf 1}_{1\times M}.
    \label{eq:systemmatrixnot}
\end{equation}
\end{proposition}
\begin{proof}
For each interval, $[t_{\ell-1}, t_{\ell}]$ consider the expansion \eqref{eq:trsol-0}, denoting the appropriate weights by $w_{j}^{(\ell)}$, $j=1, \cdots, N$, and use it to calculate the ODE at each of  
the collocation points $c_{i} \in (t_{\ell-1}, t_{\ell}]$
which upon defining  $\psi_{j}(t,\alpha^{(\ell)}_j)$ as in \eqref{matrix-form-2025},
reduces for all $i=1, \cdots, M$ to
\begin{equation}
E_{i}:=\hat{u}'(c_{i}) - \lambda \hat{u}(c_{i}) = 
\sum_{j=1}^{N} w^{(\ell)}_{j} \psi_{j}(c_{i},\theta_{j}^{(\ell)}) - \lambda \hat{u}_{\ell-1} = [(\Psi^{(\ell)})^{T} w^{(\ell)} -\lambda \hat{u}_{\ell-1}{\mathbf 1}_{M \times 1}  ]_{i},
\label{eq:expectedloss}
\vspace{-1mm}
\end{equation}
where $\Psi^{(\ell)}=(\psi_{ji}^{(\ell)})$ with $\psi_{j i}^{(\ell)}:=\psi_{j}(c_{i},\theta^{(\ell)}_j)$ and for any vector $Q \in \R^{N}$ by $[Q]_{i}$ we denote its $i$-th coordinate.
In terms of the above notation, for every $\omega \in \Omega$, we obtain from \eqref{LOSS-2025} that
\begin{equation}
{\cal L}_{\delta}^{(\ell)}(w, \omega) = \frac{1}{2} \|E \|_{{\mathbb R}^M}^2 + \frac{\delta}{2} \|w^{(\ell)} \|_{{\mathbb R}^{N}}^2= \frac{1}{2} \big\| (\Psi^{(\ell)})^{T} w^{(\ell)} - \lambda \hat{u}_{\ell -1} {\mathbf 1}_{M \times 1} \big\|_{{\mathbb R}^{M}}^2 + \frac{\delta}{2} \| w^{(\ell)} \|_{{\mathbb R}^{N}}^2,
\end{equation}
 The solution to the problem of minimizing ${\mathbb E}[{\cal L}_{\delta}(w, \omega)]$ can be obtained by pointwise solving the problem for each $\omega \in \Omega$. Note that for $\delta =0$
this problem reduces to a linear least squares problem.
For the case that $\delta >0$, we proceed with the first order conditions of the function ${\cal L}_{\delta}(w , \omega)$.
Differentiate with respect to $w^{(\ell)}_q$, $q=1, \cdots, N$, we get:
\vspace{-1mm}
\begin{equation}
\frac{\partial {\cal L}_{\delta}}{\partial w^{(\ell)}_q} = \sum_{i=1}^{M} \bigg( \sum_{j=1}^{N} w^{(\ell)}_{j} \psi_{j i}^{(\ell)} - \lambda 
\hat{u}_{\ell-1}  \bigg) \psi_{q i}^{(\ell)} + \delta w_{q }^{(\ell)} =0, \,\,\, q=1, \cdots, N.
\end{equation}
This can be expressed compactly in terms of the matrix $\Psi^{(\ell)} \in \R^{N \times M} $
as
\vspace{-1mm}
\begin{equation}
\bigg(\Psi^{(\ell)} (\Psi^{(\ell)} )^{T}  + \delta  I\bigg) w^{(\ell)}= \Psi^{(\ell)} {\bf 1}_{M \times 1} \lambda \hat{u}_{\ell -1},
\end{equation}
which yields a solution for the regularized problem in terms of
\begin{eqnarray}
\label{athens-2024}
w^{(\ell)}_{N \times 1} =((\Psi^{(\ell)})^{T})^{\dagger}_{\delta} {\bf 1}_{M \times 1} \lambda \hat{u}_{\ell -1}, \quad ((\Psi^{(\ell)})^{T})^{\dagger}_{\delta} := \big[\Psi^{(\ell)} (\Psi^{(\ell)} )^{T}  + \delta  I\big]^{-1} \Psi^{(\ell)}.
\end{eqnarray}
Note that this result holds even if $\Psi^{(\ell)} (\Psi^{(\ell)} )^{T}$  
is not invertible. 
Hence, 
\vspace{-1mm}
\begin{equation}
\begin{aligned}
u(t)=
\hat{u}_{\ell -1} \bigg( 1 +  \lambda (t - t_{\ell-1}) (\Phi^{(\ell)})^{T}(t)  \bigg[\Psi^{(\ell)} (\Psi^{(\ell)} )^{T}  + \delta  I \bigg]^{-1} \Psi^{(\ell)}    {\bf 1}_{M \times 1}  \bigg).
\end{aligned}
\end{equation}
The connection with the Moore-Penrose pseudo inverse comes by recalling the well known relation that for any matrix $A$:
\begin{eqnarray}\label{11-7-2025-2}
A^{\dagger}=\lim_{\delta \to 0^{+}} (A^{T} A + \delta I)^{-1} A^{T}= \lim_{\delta \to 0^{+}} A^{T}(A A^{T} + \delta I)^{-1}.
\end{eqnarray}
The stated result arises by taking the transpose of \eqref{athens-2024}.
\end{proof}
\subsubsection{Stability analysis of PI-RPNNs for linear scalar ODEs}
We will first prove the following theorem.

\begin{theorem}\label{PAGKRATI-STABILITY-THM}  Set the parameters in the PI-RPNN  \eqref{eq:trsol-0} such that $\epsilon:=a_{U}(\lambda, h, M) h^2$ is sufficiently small. In the limit as $a_{U}(\lambda, h, M) h^2 \to 0$, the scheme \eqref{eq:trsol-0} is asymptotically stable a.s. $\forall \lambda h <0$.
%
\end{theorem}

\begin{remark}
There is a certain freedom in the choice of $a_{U}$ which can guarantee that $\epsilon$ is small, even for $h$ quite large,  so that the perturbative approach adopted for the proof of this result is valid. A suitable choice, can be \begin{equation}
a_U=\frac{|\lambda | h^{-1}}{N + |\lambda| h + (|\lambda| h)^2+(|\lambda| h)^3}, \quad M<<N.
\end{equation}
\end{remark}


\begin{proof}
We will use the results of Proposition \ref{KALAMATA-PROPOSITION}
for the stability index, 
by considering the Gaussian RBFs as defined in \eqref{eq:trsol-0}, setting the shape parameters as stated above,

Rewriting the centers of the RBFs and the other parameters as:
\begin{equation}
\vspace{-2mm}
\tau_j^{(\ell)}=t_{\ell-1}+\xi_j^{(\ell)}h, \,\,\, \xi_{j}^{(\ell)}=\frac{j}{N}, \,\,\, \zeta_{i}^{(\ell)}=\frac{i}{M},  \,\,\, \gamma_{ji} = (\zeta_i^{(\ell)} -\xi_j^{(\ell)}) = \bigg( \frac{i}{M} - \frac{j}{N}\bigg).
    \label{eq:deftau}
\end{equation}
We now perform a careful expansion of the stability index $S_{i}$ in terms of the small parameter $\epsilon := a_U h^2$.  For ease of notation we drop  superscript $\ell$ as well as the explicit notation $a_U(\lambda, h, N)$, emphasizing however, that this dependence is still taken into account.

We next have the following expansions in terms of $\epsilon$ for the matrices  $\Phi=(\phi_{ji})$, $\Phi=(\phi_{ji})$, and   $K=(K_{j_1 j_2}) = \Psi \Psi^{T}$ 
\begin{equation}\label{16-5-2025-AAA}
\begin{aligned}
\phi_{ji} =\exp(- a_{U} h^2 \,\frac{1}{2}  \theta_j \gamma_{ji}^2 )= 1 - a_{U} h^2 \,\frac{1}{2}  \theta_j \gamma_{ji}^2 + O((a_U h^2)^2) = 1 -\epsilon \frac{1}{2}  \theta_j \gamma_{ji}^2 + O(\epsilon^2),\\
\psi_{ji}
=\underbrace{\bigg( 1- \lambda h \zeta_{i}\bigg)}_{\psi_{ji}^{(0)}} + \epsilon \underbrace{ \bigg( -  \theta_j \zeta_i \gamma_{ji} - \frac{1}{2}  \theta_j \gamma_{ji}^2 (1- \lambda h \zeta_i)  \bigg) }_{\psi_{ji}^{(1)}} +  \epsilon^2 \underbrace{  \frac{1}{2}  \theta_j^2 \zeta_i \gamma_{ji}^2 }_{\psi_{ji}^{(2)}} + O(\epsilon^3)\\
K_{j_1 j_2} = \sum_{\ell =1}^{ M} \psi_{j_1 \ell} \psi_{\ell j_2}^{T} =\sum_{\ell =1}^{ M} \psi_{j_1\ell} \psi_{j_2 \ell } 
= \underbrace{ \sum_{\ell=1}^{M} \psi_{j_1\ell}^{(0)}\psi_{j_2\ell}^{(0)} }_{K_{j_1 j_2}^{(0)}}  + \epsilon \underbrace{ \bigg( \sum_{\ell=1}^{M} \bigg[  \psi_{j_1\ell}^{(1)} \psi_{j_2 \ell}^{(0)} + \psi_{j_1 \ell}^{(0)} \psi_{j_2 \ell}^{(1)}  \bigg]  \bigg)  }_{K_{j_1 j_2}^{(1)}} + O(\epsilon^2)
\end{aligned}
\end{equation}

Introducing these expansions in the stability index yields
\begin{equation}\label{12-3-2025-17-4-2025-aaa2-AAA}
S_{i} = 1 + \lambda h \zeta_i \bigg( \Phi_i^{(0)T} + \epsilon \Phi^{(1)T} + O(\epsilon^2) \bigg) \bigg ( K^{(0)} + \delta I  + \epsilon K^{(1)} + O(\epsilon^2) \bigg)^{-1} \bigg( \Psi^{(0)} + \epsilon \Psi^{(1)} +O(\epsilon^2)\bigg) {\bf 1}_{M \times 1}.
\end{equation}
We aim to pass to the limit as $\epsilon \to 0$. This, however requires a careful consideration of the terms $ \epsilon \Phi^{(1)T}$, $\epsilon \Psi^{(1)}$ and $\epsilon K^{(1)}$, taking into account that the terms multiplying $\epsilon$ may be large due to their dependence of parameters such as $\lambda$, $M$, $N$.
To this end, before taking the limit $\epsilon \to 0$ in \eqref{12-3-2025-17-4-2025-aaa2-AAA} we perform a careful analysis of this expression, in particular with respect to parameters which are potentially large $\lambda, M, N$, and check the validity of the proposed expansion. In doing so we will also examine terms of higher order  than those needed for the zeroth order result.
We start by expressing
 $\psi_{ji}$ in terms of
\begin{equation}
\begin{aligned}
\psi_{ji}
= +\lambda h \bigg\{  - \zeta_{i} \bigg( 1 - \frac{\epsilon}{2}\theta_{j} \gamma_{ji}^2  \bigg) + (\lambda h)^{-1}  \bigg(1  - \epsilon \theta_j\zeta_i \gamma_{ji} - \epsilon \frac{1}{2}\theta_j \gamma_{ji}^2   \bigg) \bigg\} + O(\epsilon^2)
\end{aligned}
\end{equation}
so that
\begin{equation}\label{11-7-2025-1}
\begin{aligned}
K^{(0)}= M (\lambda h)^2 \bigg\{  {\mathfrak S}_{2}^{(0)} + (\lambda h)^{-1}  {\mathfrak S}_{2}^{(1)} + (\lambda h)^{-2}  {\mathfrak S}_{2}^{(2)}    \bigg\} {\bf 1}_{N \times N} =: M (\lambda h)^2 {\mathfrak S}_{2}  {\bf 1}_{N \times N}, \\
{\mathfrak S}_{2}^{(0)} =   \frac{1}{3}\bigg( 1 + \frac{3}{2} \frac{1}{M} + \frac{1}{2 M^2}  \bigg),  \quad
{\mathfrak S}_{2}^{(1)}= 
  - \bigg( 1 + \frac{1}{M} \bigg), \quad
{\mathfrak S}_{2}^{(2)}=1,  
\\
{\mathfrak S}_2 := {\mathfrak S}_{2}^{(0)} + (\lambda h)^{-1} {\mathfrak S}_{2}^{(1)} + (\lambda h)^2 \ge 0, \,\,\, \forall \,\, M, \lambda, h.
\end{aligned}
\end{equation}
Concerning the next order term in $K$ we have
\begin{equation}
\begin{aligned}
K_{j_1 j_2}^{(1)}
=- M (\lambda h)^2 \underbrace{\bigg( \theta_{j_1}\bigg\{ e_{j_1}^{(0)} + (\lambda h)^{-1} e_{j_1}^{(1)} + (\lambda h)^{-2} e_{j_1}^{(2)} \bigg\} +  \theta_{j_2}\bigg\{ e_{j_2}^{(0)} + (\lambda h)^{-1} e_{j_2}^{(1)} + (\lambda h)^{-2} e_{j_2}^{(2)} \bigg\} \bigg) }_{\bar{K}^{(1)}_{j_1 j_2} }
\end{aligned}
\end{equation}
where
\begin{equation}
\begin{aligned}
 e_{j}^{(0)}:=   \bigg(1 + \frac{1}{M}\bigg) \bigg\{+\frac{1}{10}(1 + \frac{1}{2 M} ) ( 1 + \frac{1}{M}- \frac{1}{3 M^2}) - \frac{1}{4} \,\frac{j}{N}\,\left(1 + \frac{1}{M}\right)
+\frac{1}{6}\, \bigg(\frac{j}{N}\bigg)^2 \,(1 + \frac{1}{2 M}) \bigg\}, \\
 e_{j}^{(1)}:=\left[1 + \frac{1}{M}\right] \bigg\{-\,\frac{1}{4}\,\bigg(1 + \frac{1}{M}\bigg)
+ \,  \frac{j}{N}  \bigg( 1 + \frac{1}{2 M}\bigg) 
-\frac{1}{2} \, \bigg(\frac{j}{N} \bigg)^2 \,  \bigg\}, \\
 e_{j}^{(2)}:=   \frac{1}{2}\bigg(1 +  \frac{1}{M}\bigg)  \frac{1}{2}\bigg( 1 + \frac{1}{2 M} \bigg)
- \frac{j}{N}\bigg(1 + \frac{1}{M} \bigg) 
+ \frac{1}{2}\bigg(\frac{j}{N}\bigg)^2
\end{aligned}
\end{equation}
or in more compact matrix form $K^{(1)}=-M (\lambda h)^2 \bar{K}^{(1)}$, where $\bar{K}^{(1)}=(\bar{K}^{(1)}_{j_1 j_2})$.

We emphasize that for any $M$, and $\lambda h$ the terms ${\mathfrak S}_{2}^{\infty, k}$, $e_{j}^{\infty, k}$, $k=0, 1, 2$ are $O(1)$, and likewise holds for $\bar{K}^{(1)} = O(1)$. Moreover, there is a common factor of $M$, which multiplies combinations of the above terms. This fact yields some interesting cancelation of the dependence of $M$ allowing the results to hold for any $M$. Indeed, note that
(see \eqref{11-7-2025-1})

\color{black}
\begin{equation}
\begin{aligned}
K  + \delta I= M (\lambda h)^2 \bigg[ {\mathfrak S}_{2}  {\bf 1}_{N \times N}  - \epsilon \bar{K}^{(1)} + \frac{\delta}{M (\lambda h)^2 } I \bigg],
\end{aligned}
\end{equation}
so that
\begin{equation}\label{14-6-2025-1}
\begin{aligned}
(K  + \delta I)^{-1}= M^{-1} (\lambda h)^{-2} \bigg[   {\mathfrak S}_{2} {\bf 1}_{N \times N}  - \epsilon \bar{K}^{(1)} + \frac{\delta}{M (\lambda h)^2}  I\bigg]^{-1},
\end{aligned}
\end{equation}
\color{black}
We now carefully calculate the elements of the vector $(\Psi^{(0)}+ \epsilon \Psi^{(1)} ) {\bf 1}_{M \times 1}$. 
For the zeroth order tem we have
\begin{equation}
\begin{aligned}
(\Psi^{(0)} {\bf 1}_{M \times 1})_{j} = \sum_{i=1}^{M}(1- \lambda h \zeta_i) =
M \lambda h \bar{J}_{0}, \\
\bar{J}_{0}= -\bigg[  \frac{1}{2}(1+ \frac{1}{M}) -  (\lambda h)^{-1} \bigg],
\end{aligned}
\end{equation}
and it is important to note that the term $\bar{J}_{0}$  is O(1) for any $M$ and any $\lambda h$ no matter how large it is.  
Note that there is no dependence on $j$ above so 
\begin{equation}
\Psi^{(0)} {\bf 1}_{M \times 1} = M \lambda h  \bar{J}_{0} {\bf 1}_{N \times 1}.
\end{equation}

A similar, yet more tedious calculation yields that
\begin{equation}
\begin{aligned}
(\Psi^{(0)} {\bf 1}_{M \times 1})_{j}  = M \lambda h  \bigg\{ (\lambda h)^{-1} \bar{V}_{1, j} +(\lambda h)^{-1} \bar{V}_{2,j} + \bar{V}_{3,j} \bigg\}
\end{aligned}
\end{equation}
where
\begin{equation}
\begin{aligned}
\bar{V}_{1,j}= \theta_j \bigg(1 + \frac{1}{M} \bigg) \bigg\{ -\frac{1}{3}\bigg( 1 + \frac{1}{2M} \bigg) + \frac{1}{2} \frac{j}{N} \bigg\}, \\
\bar{V}_{2,j}=\frac{1}{2} \theta_{j}\bigg\{ -\frac{1}{2} \bigg(1 + \frac{1}{M}\bigg) \bigg(1 + \frac{1}{2 M} \bigg) + \frac{j}{N} \bigg( 1 + \frac{1}{M}\bigg) -\bigg( \frac{j}{N}\bigg)^2 \bigg\}, \\
\bar{V}_{3,j}=\theta_j\bigg( 1 + \frac{1}{M} \bigg) \bigg\{ \frac{1}{2} \bigg( 1 + \frac{1}{M} \bigg)  - \frac{1}{3}\frac{j}{N} \bigg( 1 + \frac{1}{2 M}\bigg) + \frac{1}{4} \bigg( \frac{j}{N}\bigg)^2 \bigg\}.
\end{aligned}
\end{equation}
We define the corresponding vectors $\bar{V}_{k}=(\bar{V}_{k, j}, \,\, j=1, \cdots, N)$, $k=1,2,3$ as well as their sum $\bar{V}=\sum_{k=1}^{3} \bar{V}_{k}$.
Although stated explicitly, the exact form of these coefficients in immaterial here, but it is important to note that $\bar{V}_{k, j}$, $k=1,2,3$ are all $O(1)$ for any $M$ and any $j=1, \cdots, N$. Likewise for the vector $\bar{V}$ defined above. 

We conclude that
\begin{equation}\label{14-6-2025-2}
\begin{aligned}
 \bigg( \Psi^{(0)} + \epsilon \Psi^{(1)} \bigg) {\bf 1}_{M \times 1} = M \lambda h \bigg( \bar{J}_{0} {\bf 1}_{N \times n} + \epsilon \bar{V} \bigg)
\end{aligned}
\end{equation}
so that using \eqref{14-6-2025-1} and \eqref{14-6-2025-2}
\begin{equation}
\begin{aligned}
S_{i} = 1 + \lambda h \zeta_i \bigg( \Phi_i^{(0)T} + \epsilon \Phi^{(1)T} \bigg) \bigg ( K^{(0)} + \delta I  + \epsilon K^{(1)} \bigg)^{-1} \bigg( \Psi^{(0)} + \epsilon \Psi^{(1)} \bigg) {\bf 1}_{M \times 1} \\
=1 + \zeta_i \bigg( \Phi_i^{(0)T} + \epsilon \Phi^{(1)T} \bigg) \times  
\bigg[  {\mathfrak S}_{2} {\bf 1}_{N \times N}  - \epsilon \bar{K}^{(1)} + \frac{\delta}{M (\lambda h)^2}  I\bigg]^{-1} \times  \bigg( \bar{J}_{0} {\bf 1}_{N \times 1} + \epsilon \bar{V} \bigg),
\end{aligned}
\end{equation}
where all terms  multiplying $\epsilon$ are now of $O(1)$ for  any choice of $M$ and $\lambda h$ (as long as it is bounded away from $0$).

We can now pass to the limit as $\epsilon \to 0$. Since $\epsilon \bar{K}^{(1)}, \epsilon \Phi_{i}^{(1)T}, \epsilon \bar{V} \to 0$ in this limit, we obtain that as $\epsilon \to 0$, the stability index tends to
\begin{equation}
\begin{aligned}
S_{i} = 1 + \zeta_i {\bf 1}_{1 \times N} \bigg[  {\mathfrak S}_{2}  \underbrace{{\bf 1}_{N \times N} }_{ {\bf 1}_{1 \times N}^{T} {\bf 1}_{1 \times N}}   + \frac{\delta}{M (\lambda h)^2}  I\bigg]^{-1} \times   \bar{J}_{0} {\bf 1}_{N \times 1}  \\
\end{aligned}
\end{equation}

We now take the limit as 
$\delta \to 0$ (or $M \to \infty$ or $|\lambda h| \to \infty$) using \eqref{11-7-2025-2} for $A= {\mathfrak S}_2^{1/2} 1_{1 \times N}^{T}$ which yields
\begin{equation}
\begin{aligned}
S_i = 1 + \zeta_i {\mathfrak S}_2^{-1} \bar{J}_0 {\bf 1}_{N \times 1}^{\dagger} {\bf 1}_{N \times 1} = 1 + \zeta_i {\mathfrak S}_2^{-1} \bar{J}_0 
=  1 - \zeta_i\frac{\bigg\{ \frac{1}{2}(1+ \frac{1}{M}\bigg) - (\lambda h)^{-1} \bigg\}  }{\bigg\{  \frac{1}{3} \bigg( 1 + \frac{3}{2} \frac{1}{M} + \frac{1}{2 M^2}  \bigg) - (\lambda h)^{-1} \bigg( 1 + \frac{1}{M} \bigg) + (\lambda h)^{-2}  \bigg\}   }
\end{aligned}
\end{equation}

Since the denominator always keeps it sign it can be seen that  for $\lambda >0$ and such that $(\lambda h)^{-1} > \frac{1}{2} ( 1 + \frac{1}{M})$ the scheme is unstable as expected.

On the other hand for $\lambda < 0$, the scheme is always stable since in this case setting $s=(|\lambda| h)^{-1}$ we see that $S_i < 1$ but at the same time $S_{i} > -1$ for all $i=1, \cdots, M$. This last on can be seen since
\begin{equation}
\begin{aligned}
-1 < S_{i} \,\, \Longleftrightarrow \,\, \zeta_i \frac{\bigg\{ \frac{1}{2}(1+ \frac{1}{M}\bigg) - (\lambda h)^{-1} \bigg\}  }{\bigg\{  \frac{1}{3} \bigg( 1 + \frac{3}{2} \frac{1}{M} + \frac{1}{2 M^2}  \bigg) - (\lambda h)^{-1} \bigg( 1 + \frac{1}{M} \bigg) + (\lambda h)^{-2}  \bigg\}   } < 2
\end{aligned}
\end{equation}
which will hold for all $i=1, \cdots, M$ as long as it holds for $i=M$ as it is equivalent to
\begin{equation}
\begin{aligned}
 0 < 
\bigg( \frac{1}{3}  + \frac{2}{3} \frac{1}{M} + \frac{1}{3 M^2} \bigg) 
+ s \bigg( 1 + \frac{2}{M} \bigg) + 2 s^2,
\end{aligned}
\end{equation}
which holds for all $s=|\lambda h| > 0$ and $M \in {\mathbb N}$.
\end{proof}

\color{black}
\begin{remark}
As also demonstrated by numerical simulations depicted in Figure \ref{fig:absolute_stability_region_multicoll} for various choices of $N$ and $M$, the multi-collocation RPNN scheme is A-stable (recall that a numerical method for an ODE is called \textit{A-stable} if its region of absolute stability includes the entire left half of the complex plane).
\end{remark}
\begin{figure}[ht!]
\vspace{-2mm}
    \centering
    \subfigure[$M=4, N=12$]{\includegraphics[trim={0cm 0cm 1cm 0.5cm},clip,width=0.31\textwidth]{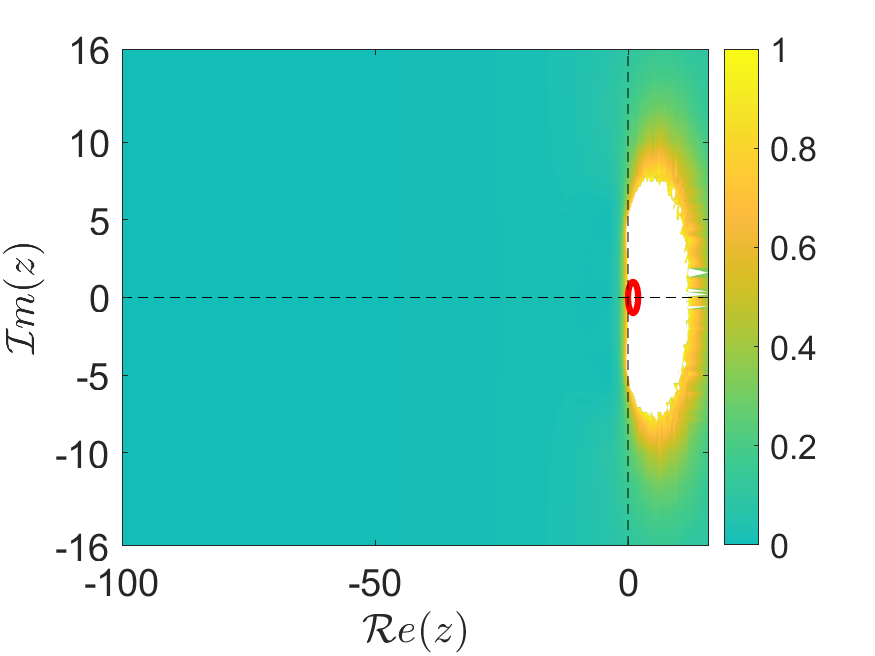}}
    \subfigure[$M=10,N=30$]{\includegraphics[trim={0cm 0cm 1cm 0.5cm},clip,width=0.31\textwidth]{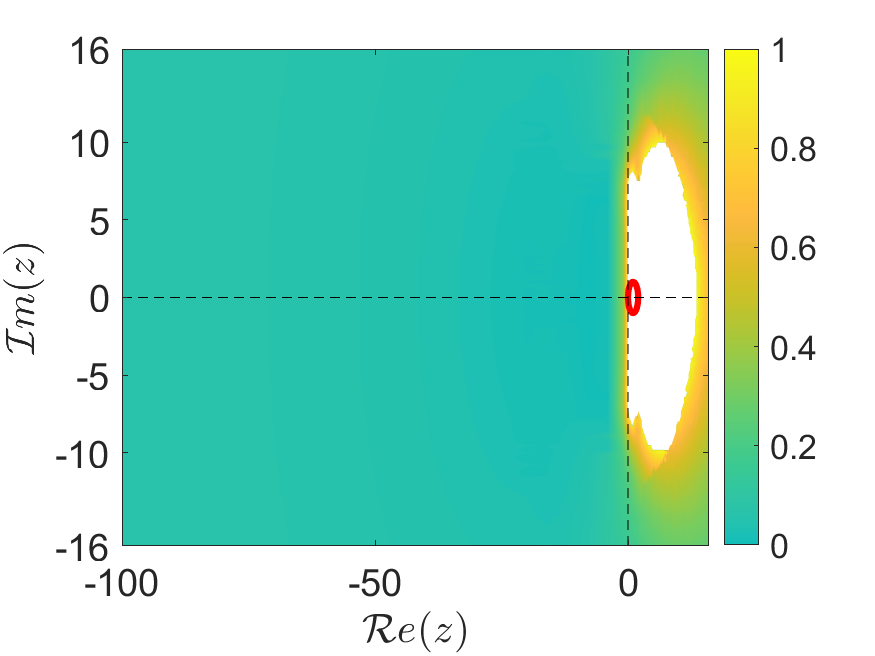}}
    \subfigure[$M=50, N=150$]{\includegraphics[trim={0cm 0cm 1cm 0.5cm},clip,width=0.31\textwidth]{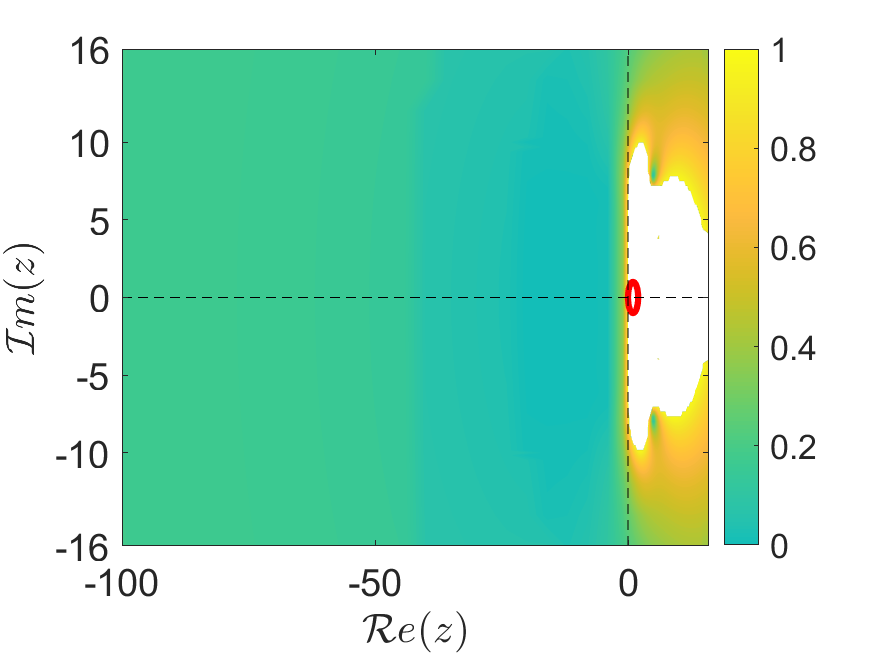}}
    \subfigure[$M=4, N=12$ (zoom)]{\includegraphics[trim={0cm 0cm 1cm 0.5cm},clip,width=0.31\textwidth]{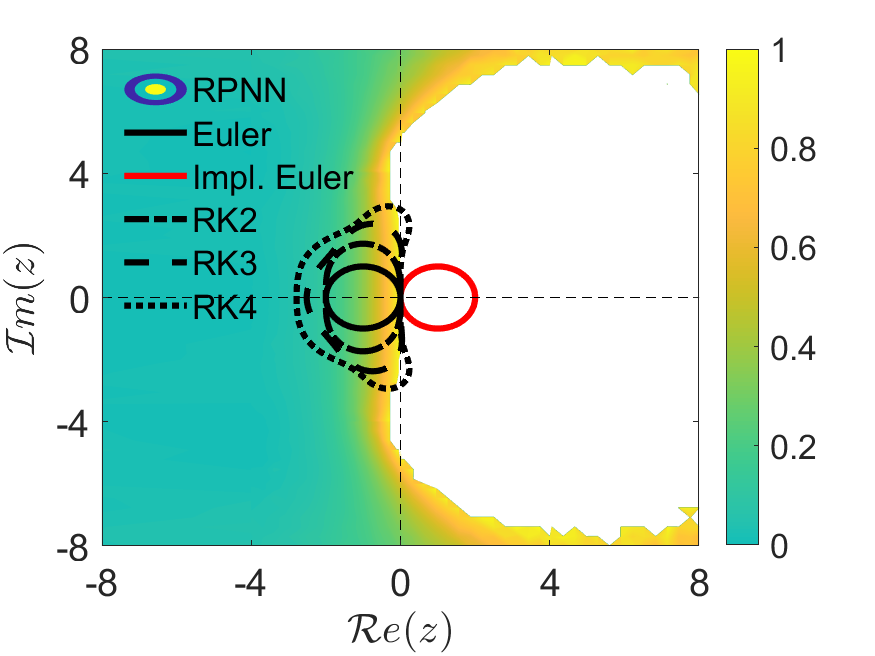}}
    \subfigure[$M=10,N=30$ (zoom)]{\includegraphics[trim={0cm 0cm 1cm 0.5cm},clip,width=0.31\textwidth]{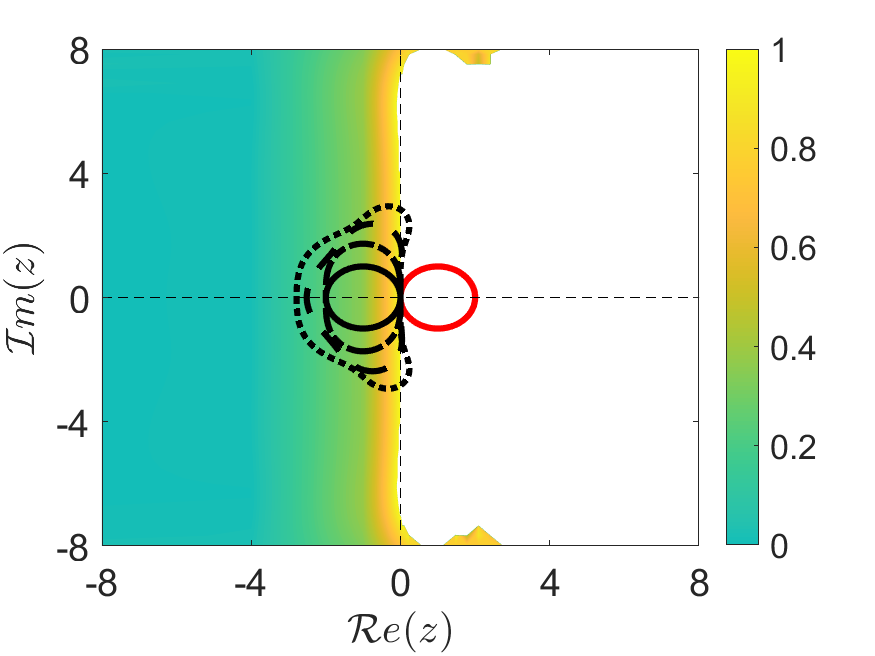}}
    \subfigure[$M=50, N=150$ (zoom)]{\includegraphics[trim={0cm 0cm 1cm 0.5cm},clip,width=0.31\textwidth]{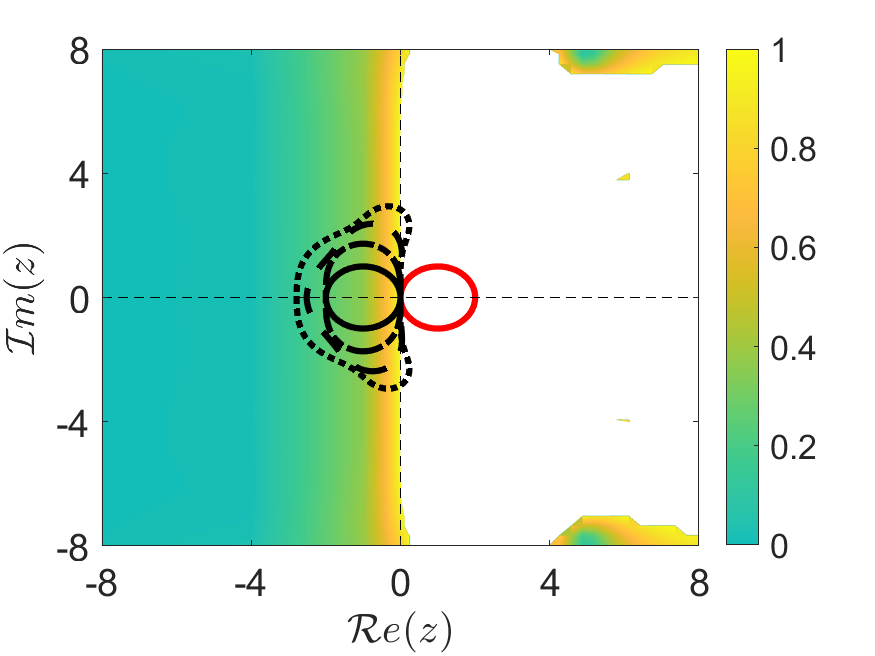}} \vspace{-3mm}
    \caption{Multi-collocation physics-informed RPNN for the solution of the scalar linear ODE. Absolute, stability region of the multicollocation RPNN (colored zone), with numerical simulations.
    We set in (a),(d) $M=4$; in (b),(e) $M=10$; in (c),(f) $M=50$; In all cases we fixed $N=3M$.
    We report the absolute stability function $|S|$ in the range $[0,1]$. The white zone indicates $|S|>1$ (unstable region). We report maximum values of the stability function, 
    over $200$ Monte-Carlo runs. A comparison with the stability domains of explicit and implicit Euler, and Runge-Kutta methods of orders 2, 3 and 4 (contour lines) are also given. 
    We used a non-uniform mesh with approximately 100 by 100 points covering the complex plane region in  $\big\{z=(\mathcal{R}e(z),\mathcal{I}m(z))[-100,16]\times[-16,16]\big\}$, with increased density near zero and near the boundaries of the stability region in both real and imaginary parts, and extending far left to capture high stiffness.       \label{fig:absolute_stability_region_multicoll}}
\end{figure}

\subsubsection{Consistency of PI-RPNNs for linear scalar ODEs}
\begin{theorem} Let $u(\cdot)$ be the exact solution of \eqref{graf-2025} with $\lambda<0$,  and $\hat{u}(\cdot)$ be the approximate solution provided by the multi-collocation RPNN
 given by \eqref{eq:trsol-0} with weights determined by \eqref{OPT-24-2-205}. Then,
\begin{eqnarray*}
{\cal E}_{i}=u_{\ell}(c_i) - \hat{u}_{\ell}(c_i) \to 0, \,\,\, \mbox{as} \,\,\, h \to 0, \forall \delta >0, M.
\end{eqnarray*}
Moreover, in the limit as $\delta \to 0$ or $M \to \infty$ it satisfies the property
\vspace{-2mm}
\begin{eqnarray*}
\frac{{\cal E}_i}{h} \to 0 \,\,\, \mbox{as} \,\,\, h \to 0,
\end{eqnarray*}
i.e. it is consistent for the solution of the IVP: $\frac{du(t)}{dt}=\lambda u, \quad u(t_0)=u_0$ in the limit as $M \to \infty$.
\end{theorem}
\begin{proof}
We will check consistency at any collocation point. In order to ease the notation, and since we are focusing on a single interval $[t_{\ell-1}, t_{\ell}]$ we will omit the supercripts $(\ell)$ from the representation \eqref{peroni-1000}, which now reads
\begin{equation}
\begin{aligned}
\hat{u}_{\ell}(c_i)=\big (1+ \lambda h \zeta_i \Phi_i^{T}(\Psi \Psi^T + \delta I)^{-1} \Psi {\bf 1}_{M \times 1} \big) u(t_{\ell -1}),
\end{aligned}
\end{equation}
where as mentioned above we assume that at $t_{\ell -1}$ we have the true solution $u(t_{\ell -1})$ and not the approximate solution $\hat{u}_{\ell -1}$ (as in \eqref{peroni-1000}).

We now estimate the error of the method at each collocation point by using
\begin{equation}\label{peroni-1001}
\begin{aligned}
{\cal E}_{i}= u_{\ell}(c_{i}) - \hat{u}_{\ell}(c_i) = 
u_{\ell}(c_{i}) -\big (1+ \lambda h \zeta_i \Phi_i^{T}(\Psi \Psi^T + \delta I)^{-1} \Psi {\bf 1}_{M \times 1} \big)  u_{\ell -1}.
\end{aligned}
\end{equation}
We now use Taylor's theorem around to point $t_{\ell -1}$ to obtain for the true solution that
\begin{equation}\label{peroni-1002}
u_{\ell}(c_i)= u(t_{\ell -1}) +  h \zeta_i u'_{\ell -1} + O(h^2) =
u_{\ell-1} + \lambda h \zeta_i u_{\ell -1} + O(h^2),
\end{equation}
where we used the fact that $u$ solves the ODE, and the regularity of the solution.  

Combining \eqref{peroni-1001} and \eqref{peroni-1002} we obtain that
\begin{equation}\label{peroni-1003}
\begin{aligned}
{\cal E}_{i}
= \lambda  h \zeta_i   \big(\underbrace{ 1 - \Phi_i^{T}(\Psi \Psi^T + \delta I)^{-1} \Psi {\bf 1}_{M \times 1} }_{:=A(h) }\big)\  + O(h^2).
\end{aligned}
\end{equation}
As long as the limit $\lim_{h \to 0^{+}} A(h)$ is finite then we can easily see from \eqref{peroni-1003} that $\lim_{h \to 0^{+}} {\cal E}_{i}=0$ from which consistency follows. 

To show that $A(h) \to c$, finite, in the limit as $h \to 0^{+}$, we work in the same spirit as in the proof of Theorem \ref{PAGKRATI-STABILITY-THM}. To facilitate the proof and the exposition we set $z=\lambda h$ and keeping $\lambda $ finite we take the limit as $h \to 0^{+}$, i.e. consider the limit as $z \to 0$ (we will allow $\lambda$ being either positive or negative, as we want to show consistency in the general case).
In the limit as $z \to 0^{\pm}$ we easily see that
\begin{equation}\label{peroni-3000}
\begin{aligned}
\Phi_{i}^{T}(z) \to {\bf 1}_{1 \times N}, \,\,\,\,\,
\Psi(z) \to {\bf 1}_{N \times M},
\end{aligned}
\end{equation}
We then express $A(h)$ as 
\begin{equation}
\begin{aligned}
A(z)=1 - \Phi_{i}^{T}(z) (\Psi(z)\Psi^{T}(z) + \delta I)^{-1} \Psi(z) {\bf 1}_{M \times 1}.
\end{aligned}
\end{equation}
Note that, using the notation $a(z)=h a_{U}(\lambda, h)$, where $a(z)$ is chosen so that $a(z) \to 0$ as $z \to 0$, from \eqref{peroni-3000} and the continuity of $z \to (\Psi(z) \Psi^{T}(z)+ \delta I)^{-1}$ we get
\begin{equation}
\begin{aligned}
A(z) \to A_0(\delta, M):= 1 -{\bf 1}_{1 \times N} ( {\bf 1}_{N \times M} {\bf 1}_{M \times N} + \delta I )^{-1} {\bf 1}_{N \times M}{\bf 1}_{M \times 1} 
\\
=1 - {\bf 1}_{1 \times N} \big( {\bf 1}_{N \times N} + \frac{\delta}{M}\big)^{-1} {\bf 1}_{N \times 1},
\end{aligned}
\end{equation}
which is definitely finite. Hence, \eqref{peroni-1003} yields that ${\cal E}_{i}(h) \to 0$ in the limit as $h \to 0^{+}$. 
In the limit as $\delta \to 0$ or $M \to \infty$, we see that  $A_0(\delta, M) \to 0$, and our second claim holds.
By similar arguments we may show that the same result holds for every $t \in [t_{\ell-1}, t_{\ell}]$, and not just on the collocation points (consistency in the uniform norm). 
\end{proof}
\subsection{Linear stability analysis of a system of ODEs}
\label{sec:system_stability}
We will now extend our stability analysis  to systems of ODEs. 
We will first provide a partial result for the multi-collocation scheme for the case where $A$ is diagonalizable (Theorem \ref{DIAG-T}), that will be of interest in its own right for the extension of our results to the PDE case.
\begin{proposition}\label{DIAG-T}
Consider the ODE system \eqref{eq:simplelinear0}, assume that $A$ is diagonalizable and let $\{\lambda_i, \,\,\, \cdots, \,\, i=1, \cdots, d\}$ be the eigenvalues of $A$. Then, the RPNN multi-collocation scheme (as in Theorem \ref{PAGKRATI-STABILITY-THM})  is asymptotically stable if $\max_{i} \lambda_i h  \to  0^{-}$ and unstable if $\lambda_i h \to 0^{+}$ for at least one $i$. Also for $\max_{i} \lambda_i h \to -\infty$, the scheme is asymptotically stable.
\end{proposition}
\begin{proof}
By the diagonalizability of $A$ there exists a similarity transformation $S$ such that $A=S^{-1} D S$ where $D=diag(\lambda_1, \cdots, \lambda_d)$, with $\lambda_i$, the eigenvalues of the matrix $A$. Upon defining $\bm{z}=S \bm{u}$, we can transform system \eqref{eq:simplelinear0} to a decoupled system
\begin{eqnarray}\label{TH-31}
\frac{d \bm{z}}{dt} = D \bm{z}, \,\,\,\,\, \bm{z}(0)=\bm{z}_0 := S \bm{u}_0, \quad \Longleftrightarrow \,\,\, \frac{d z_{i}}{dt} = \lambda_{i} z_{i}, \,\,\,\, i=1, \cdots, d.
\end{eqnarray}
If $\lambda_i \le 0$ for all $i$, the stability of the scheme is dominated by the eigenvalue with minimum value of $|\lambda_i|$. 
Then applying  Theorem \ref{PAGKRATI-STABILITY-THM} 
we obtain the stated result.
\end{proof}
For the non-diagonalizable case, we state the following Theorem.
\begin{theorem}\label{GENERAL-MATRIX-T}
For a general non-diagonalizable linear system of ODEs of the \eqref{eq:simplelinear0}, where $A$ is a matrix with its eigenvalues all real and negative, the  PI-RPNN multicollocation scheme is asymptotically stable.
\end{theorem}
\begin{proof} To simplify the notation set $\alpha_j^{(\ell)}=a_{U} \theta_{j}^{(\ell)}$.

Suppose that the matrix $A$ is not diagonizable because its spectrum consists of non-single eigenvalues. In this case the matrix $A$ can be transformed in Jordan block form using the Jordan decomposition $A=PJP^{-1}$, 
\begin{eqnarray}
J=\left(\begin{array}{ccccc}
J_1 & 0 &  & \\
0 & J_2 &  & \\
  & \ddots & &  \\
  0 & &  &  J_{q}
  \end{array}\right), \,\,\,\,  J_{m}= \left( \begin{array}{cccccc}
  \lambda & 1 &  0 &  &    \\
  0 & \lambda & 1 & 0 &      \\
   &  & \ddots & &  \\
   &  &       &  \lambda & 1 \\
   &  &      &    0   &  \lambda 
   \end{array}\right)   \in {\mathbb R}^{m \times m},
  \end{eqnarray}
where each $J_{m}$ is a Jordan block of dimension $m\times m$, corresponding to the multiplicity of the corresponding eigenvalue. The invertible matrix $P$, whose columns are generalized eigenvectors of $A$, defines a new coordinate system $\bm{u}=P\bm{z}$ under which we can transform the system \eqref{eq:simplelinear0} to
\begin{equation}
\vspace{-1mm}
\frac{d \bm{z}}{dt}= J \bm{z}, \,\,\,\,\, \bm{z}(0)=\bm{z}_0 := P^{-1} \bm{u}_0.
\end{equation}
For these systems it suffices to study the stability analysis on the single Jordan block \cite{gear1984ode}.
Hence, we now consider the ODE for each Jordan block $J_{m}$,
\begin{equation}
\vspace{-1mm}
\frac{d \bm{y}}{dt}= J_{m} \bm{y}.
\label{eq:jordanblockdif}
\end{equation}
\color{black}
Setting $\bm{y}=(y_1, \cdots, y_m)^{T}$, and $y_{\ell}$, $\ell =1, \cdots, m$, we adopt the expansions
\vspace{-1mm}
\begin{eqnarray}\label{B-0}
y_{\ell}(t) =y_{\ell}(t_{i-1}) + (t-t_{i-1}) \sum_{j=1}^{M} w_{j,\ell}^{(i)} e^{-a_{j,\ell}^{(i)} (t- \tau_{j,\ell}^{(i)})^2}, \,\,\, t \in [t_{i-1},t_{i}], 
\end{eqnarray}
allowing, if necessary, different parameters for each $\ell$.
Differentiating \eqref{B-0}
 we can  obtain the errors ${\cal E}_{\ell}^{(i)}:=\frac{d y_{\ell}}{dt}(t_i)-\lambda y_{\ell}(t_i) - y_{\ell+1}(t_i) $   for each collocation point $t_{i}$. For notational convenience we define the following vectors
\begin{equation}
\begin{aligned}
q_{\ell}^{(i)}&=(e^{-a_{j,\ell}^{(i)} (t- t_{j,m}^{(i)})^2}, \,\,\, j=1, \cdots, M)  \in {\mathbb R}^{M}, \\
k_{\ell}^{(i)}&= (\big(1-\lambda h - 2 a_{j,\ell}^{(i)} h (t_{i} -\tau_{j,\ell}^{(i)}) \big)  e^{-a_{j,\ell}^{(i)} (t- \tau_{j,\ell}^{(i)})^2}, \,\,\,\, j=1, \cdots, M)  \in {\mathbb R}^{M}, \\
w_{\ell}^{(i)} &= (w_{j,\ell}^{(i)}, \,\,\,\, j=1, \cdots, M) \in {\mathbb R}^{M},
\end{aligned}
\end{equation}
noting that in the future for simplicity we may drop the explicit $i$
 dependence. In terms of the above we  have
\begin{equation}\label{20-7-2024-100}
 \begin{aligned}
{\cal E}_{m}^{(i)}&= \langle w_{\ell}^{(i)}, k_{m}^{(i)} \rangle - \alpha_{m}^{(i)}, \\
{\cal E}_{\ell}^{(i)} &= \langle w_{\ell}^{(i)}, k_{\ell}^{(i)} \rangle - h \langle w_{\ell+1}^{(i)}, q_{\ell+1}^{(i)}\rangle - \alpha_{\ell}^{(i)}, \,\,\,\, \ell =1, \cdots, m-1, \\ 
\alpha_{m}^{(i)}&=\lambda y_{m, i-1}, \,\,\, \\
\alpha_{\ell}^{(i)}&=\lambda y_{\ell, i-1} + y_{\ell+1,i-1}, \,\,\,\, \ell =1, \cdots, m-1.
\end{aligned}
 \end{equation}
The choice of the (random) weights $w_{\ell}^{(i)}$, $\ell=1, \cdots, m$, will be made so that the square error at the collocation points $\{t_{i} \}$ is minimized, i.e.
\begin{equation}\label{FL-1}
(w_{\ell}^{(i)}, \,\,\, \ell=1, \cdots, m) \in \arg\min_{(w_{\ell}^{(i)}, \,\,\, \ell=1, \cdots, m)} \sum_{\ell=1}^{m} ({\cal E}_{\ell}^{(i)})^{2}.
\end{equation}
The first order conditions reduce to ${\cal E}_{\ell} =0$, $\ell=1, \cdots, m$,
which in turn reduces to
 \begin{equation}\label{FL-4}
 \begin{aligned}
&\langle w_{m}^{(i)}, k_{m}^{(i)} \rangle = r_{m} :=\alpha_{m}^{(i)} \\
& \langle  w_{\ell}^{(i)}, k_{\ell}^{(i)} \rangle = r_{\ell}^{(i)} :=  h \langle w_{\ell+1}^{(i)}, q_{\ell+1}^{(i)} \rangle+ \alpha_{\ell}^{(i)} , \,\,\, \ell =m-1, \cdots, 1.
 \end{aligned}
 \end{equation}
Note that system \eqref{FL-4} can be solved in terms of a backward iteration scheme, starting from the equation for $k_{m}^{(i)}$, then proceeding to $\ell=m-1$ and substituting the solution obtained for $w_{m}^{(i)}$ to get $w_{m-1}^{(i)}$, and working similarly backwards up to $\ell=1$.\par
In the same fashion as for the diagonal case we have that \eqref{FL-4} yields
\vspace{-2mm}
 \begin{eqnarray}\label{FL-7}
w_{\ell}^{(i)} = \frac{r_{\ell}^{(i)}}{\| k_{\ell}^{(i)}\|^2} k_{\ell}^{(i)}, \,\,\,\,\, \ell= m, \cdots, 1,
 \vspace{-2mm}
 \end{eqnarray}
which is a solution in  implicit form since the coefficients $r_{\ell}^{(i)}$, given by the RHS of \eqref{FL-4}, depend on $w^{(i)}_{\ell+1}$.  Using however the  definition of $r_{\ell}^{(i)}$  in \eqref{FL-7}  in combination with
 \eqref{FL-7}, yields the backward iteration scheme
\begin{equation}\label{FL-9}
\begin{aligned}
& r_{m}^{(i)} = \alpha_{m}^{(i)}, \quad \\ 
&r_{\ell}^{(i)}=  \Lambda_{\ell +1} r_{\ell+1} + \alpha_{\ell}^{(i)}, \,\,\,\,\ell =m-1, \cdots, 1, \\
& \Lambda_{\ell +1}^{(i)} := h \frac{\langle  k_{\ell+1}^{(i)}, q_{\ell+1}^{(i)} \rangle}{\| k_{\ell+1}^{(i)} \|^2},
\end{aligned}
\vspace{-2mm}
\end{equation}
that can be explicitly solved in terms of
$\Lambda_{\ell}^{(i)}$, $\alpha_{\ell}^{(i)}$ and combined with \eqref{FL-7} to obtain the required solution for the weights. 

By induction and recalling the dependence of $a$ on $y$ (see \eqref{20-7-2024-100})
we obtain:
\begin{equation}\label{B-2}
    \begin{aligned}
        r_{m-\nu}^{(i)} =\lambda y_{m-\nu, i-1} + \sum_{\ell=0}^{\nu-1} \bigg( \prod_{\sigma=\ell+1}^{\nu-1} \Lambda_{m-\sigma}^{(i)} \bigg) (1+ \lambda \Lambda_{m-\ell} ) y_{m-\ell, i-1}, \,\,\,\, \nu=1, \cdots, m-1,
    \end{aligned}
\end{equation}
where  the convention that $\prod_{s=1}^{0} \Lambda_{s}^{(i)}  :=1 $ is used.

We can now reconstruct the weights $w_{\ell}^{(i)}$, using \eqref{FL-7} and \eqref{B-2}, to obtain an expression for $\bm{y}_{i}=(y_{1,i}, y_{2,i}, \cdots, y_{m,i})^{T}$ in terms of $\bm{y}_{i-1}=(y_{1,i-1}, y_{2,i-1}, \cdots, y_{m,i-1})^{T}$ using the representation \eqref{B-0}. After some algebra
we get that
\begin{eqnarray}
  \bm{y}_{i} = M^{(i)} \bm{y}_{i-1} ,
\end{eqnarray}
where $M^{(i)}$ is the upper triangular $m\times m$ matrix
\begin{equation}
M^{(i)}=\begin{pmatrix}
\ddots & \vdots & \vdots & \vdots & \vdots &  \vdots                                        \\
& \ddots & \vdots &  \vdots & \vdots &       \vdots                                     \\
& & &    (1+\lambda \Lambda_{m-2}^{(i)} )         & \Lambda_{m-2}^{(i)}  (1+\lambda \Lambda_{m-1}^{(i)} )         & \Lambda_{m-2}^{(i)}    \Lambda_{m-1}^{(i)}(1+\lambda \Lambda_{m}^{(i)} )       \\
 & & &  0   & (1+\lambda \Lambda_{m-1}^{(i)})  &\Lambda_{m-1}^{(i)}(1+\lambda \Lambda_{m}^{(i)})  \\
  & & &  0   & 0 & (1+\lambda \Lambda_{m}^{(i)})
\end{pmatrix}
\end{equation}
The eigenvalues of the matrix $M_{i}$ are the diagonal elements $M^{(i)}_{jj}=1 + \lambda \Lambda_{j}^{(i)}$, $j=1,\cdots, m$.

Note that as of Thm. \ref{PAGKRATI-STABILITY-THM}, we have that $|1+\lambda {\Lambda_j}^{(i)}|<1$. Recursively, we get:
\begin{equation}
   \bm{y}_i = M^{(i)} M^{(i-1)} \cdots M^{(1)} \bm{y}_0
\end{equation}
As each $M^{(j)},, j=1,\dots,i$ is an upper triangular matrix the transition matrix $M^{(i)} M^{(i-1)} \cdots M^{(1)}$ is also an upper triangular matrix with eigenvalues the product of the eigenvalues of each $M^{(j)}, j=1,\dots,i$. Hence, the 
matrix $M^{(i)} M^{(i-1)} \cdots M^{(1)}$ has $m$ distinct eigenvalues and 
the $m$ corresponding eigenvectors, say, $\{{\bm{v}_1}^{(i)},{\bm{v}_2}^{(i)},\dots {\bm{v}_m}^{(i)}\}$ form a basis in $\mathbb{R}^m$. Hence, we can write $y_{0}$ as  $\bm{y}_{0}=\sum_{j=1}^m {c_j}^{(i)} {\bm{v}_j}^{(i)}$.
Since:
\vspace{-2mm}
\begin{equation}
 M^{(i)} M^{(i-1)} \cdots M^{(1)} {\bm{v}_j}^{(i)} = \prod_{p=1}^{i}(1 + \lambda \Lambda_{j})^{p} {\bm{v}_j}^{(i)}, \, j=1,2,\dots,m,
 \vspace{-1mm}
 \end{equation}
we have:
\begin{equation}
\begin{aligned}
\bm{y}_{i} = {c_1}^{(i)} M^{(i)} M^{(i-1)} \cdots M^{(1)} {\bm{v}_1}^{(i)} + c_2 M^{(i)} M^{(i-1)} \cdots M^{(1)} {\bm{v}_2}^{(i)}   \\
+ \cdots + {c_m}^{(i)}M^{(i)} M^{(i-1)} \cdots M^{(1)} {\bm{v}_m}^{(i)}, 
\end{aligned}
\end{equation}
or
\begin{equation}
    \bm{y}_{i} = {c_1}^{(i)}  \prod_{p=1}^{i} (1 + \lambda \Lambda_{1})^{p} \bm{v}_1^{(i)} + {c_2}^{(i)} \prod_{p=1}^{i}(1 + \lambda \Lambda_{2})^{p} \bm{v}_2^{(i)} + \cdots + {c_m}^{(i)} \prod_{p=1}^{i}(1 + \lambda \Lambda_{m})^{p} \bm{v}_m^{(i)}.
\end{equation}
But each product $ \prod_{p=1}^{i} (1 + \lambda \Lambda_{j})^{p} \rightarrow 0, j=1,2,\dots,m$, as $i\rightarrow \infty$  because \\
$|(1 + \lambda \Lambda_{j})^{p}| < 1 $ $ \forall p,j$. Consequently:
\vspace{-1mm}
\begin{equation}
\lim_{i \to \infty} \bm{y}_i = \bm{0},
\vspace{-2mm}
\end{equation}
and therefore because of $\bm{u}=P\bm{z}$, and, (\ref{eq:jordanblockdif}), we conclude.
\end{proof}

\begin{remark}
Note that the above proof can be extended to every matrix $A$ over the field of complex numbers, i.e., for matrices also with multiple complex eigenvalues (see for example in \cite{sergeichuk2000canonical}). 
\end{remark}
\vspace{-2mm}
\color{black}
\begin{remark}
    The results of this section can be applied also in the case of linear parabolic equations, after discretization in space and time using standard methods, such as finite differences schemes, or using an eigenfunction expansion, in which case one can apply Proposition \ref{DIAG-T} (see example \ref{sec:lin_diff_rec} below).
\end{remark}

\section{Numerical results}
\label{sec:num_results}
In this section, we present indicatively, the numerical approximation accuracy and computational costs of the proposed RPNN.
For our illustrations, we utilize two benchmark problems included in the main text: {a 2d stiff 
system of non-normal linear ODEs \cite{higham1993stiffness},} and a linear diffusion-reaction PDE discretized using finite differences. 
Furthermore, we compare the performance of the proposed scheme versus various traditional implicit  schemes: the Backward Euler, implicit midpoint, implicit trapezoidal (Crank-Nikolson), the 2-stage Gauss scheme and the 2 and 3 stages Radau schemes. The numerical results that we present here are indicative. We don't aim to make a direct comparison on the numerical approximation accuracy and computational cost with the other schemes. We have presented an extensive comparison analysis for stiff nonlinear systems of ODEs and DAEs in another work \cite{fabiani2023parsimonious}, where the focus was on the optimal tuning of the hyperparameters of PIRPNNs (see also in the discussion).
For our illustrations, we employed 3-stages (M=3, N=9) and 10 stages (M=10 and N=30) PIRPNN schemes. The upper bounds $a_U$ are computed according to the requirements proved in the previous sections. Also to be computationally efficient, we randomize only one time, once and for all, the internal weights of the PIRPNN, and we keep them fixed along the entire time domain. In this way, we can precompute the pseudo-inverse matrix for the solution of the problem. Specifically, we compute the pseudoinverse using the standard SVD decomposition for low-dimensional ODEs (e.g., example 1), and a sparsity-exploiting COD decomposition for high-dimensional ODEs (e.g., example 2) \cite{fabiani2023parsimonious}. All the computations are carried on single core of a Intel Core i7-10750H CPU \@ 2.60GHz, with $16$GB of RAM running Matlab 2020b.

\color{black}
\subsection{Example 1: A Linear ODE System – Constant Coefficients, Non-Normal, Homogeneous}
The model reads:
\vspace{-1mm}
\begin{equation}
    \frac{dy}{dt}=Ay, \qquad A=\begin{bmatrix}
        -10 & 100\\ 0 & -1
    \end{bmatrix}.
\end{equation}
This is taken from example 2 in \cite{higham1993stiffness}. 
\begin{figure}[ht!]
    \centering
    \subfigure[]{\includegraphics[trim={0.cm 0.cm 0.5cm 0.5cm},clip,width=0.45\textwidth]{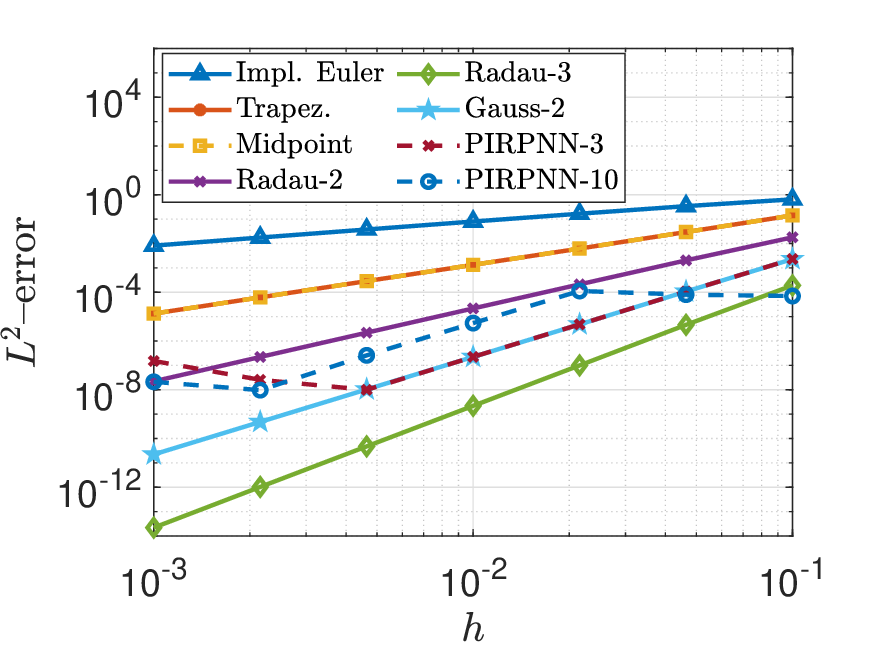}
    }
    \subfigure[]{\includegraphics[trim={0cm 0cm 0.5cm 0.5cm},clip,width=0.45\textwidth]{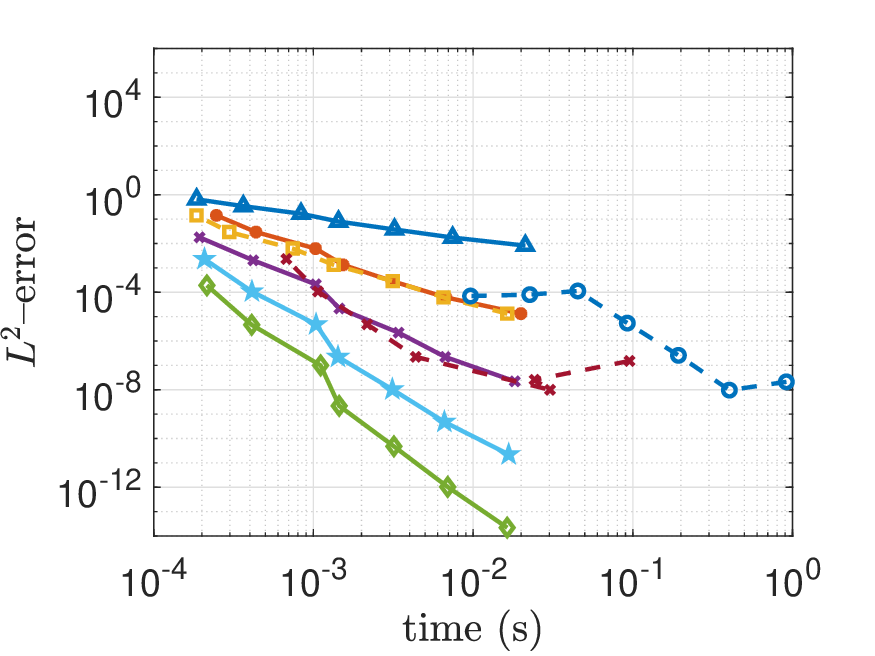}
    }   \vspace{-3mm} \caption{\color{black} Numerical approximation accuracy and computational cost for the stiff "non-normal" linear ODEs benchmark problem in \cite{higham1993stiffness}. The PIRPNNs, uses $M=3$ and $M=10$ collocation points in each subinterval of size $h$ and $N=9$ and $N=30$ neurons, respectively. The numerical approximation accuracy of the PIRPNNs are reported in dashed lines, while traditional implicit schemes with solid lines. (a) $L^2$--error in terms of fixed time step $h$. (b) $L^2$--error with respect to machine/computational execution time.}
    \label{fig:tref2d}
\end{figure}
We solve the problem in the interval $[0,5]$ using various fixed time steps, $h$. Figure \ref{fig:tref2d} compares the convergence of our proposed PIRPNN, with several traditional implicit solvers.

The three-point and ten-point PIRPNN collocation schemes exhibit the expected stability across the entire range of time steps tested. Although their accuracy is sub-optimal, due to fixed internal parameters not being tuned for precision. They still outperform classical implicit schemes such as implicit Euler, implicit trapezoidal, midpoint, and Radau-2. The three-point and ten-point PIRPNN schemes also show performance comparable to the Gauss-2 method.

\color{black}
\subsection{A linear diffusion-reaction PDE problem}
\label{sec:lin_diff_rec}
We consider a simple, linear diffusion-reaction PDE given by:
\vspace{-2mm}
\begin{eqnarray}
u_t=\nu u_{xx}-\lambda u.
\label{eq:pdelinear}
\end{eqnarray}
with Neumann BCs in $[0, \pi]$ and initial conditions
$u(x,0) = a \cos(2 x)+c$  (here $a =0.4, c =1.5 $).
Based on the above, the analytical solution $u(t,\bm{x})$ of  Eq. \eqref{eq:pdelinear} reads:
\begin{eqnarray}
u(t,\bm{x}) = a \exp(-(4\nu+\lambda) t) \cos(2 x)+c \exp(-\lambda t).
\label{eq:pdeqn9}
\end{eqnarray}
For our simulations, we have set $\nu=0.1$ and $\lambda=10$. 
\begin{figure}[ht!]
    \centering
    \subfigure[]{ 
    \includegraphics[trim={0 0 0.5cm 0.5cm},clip,width=0.45\textwidth]{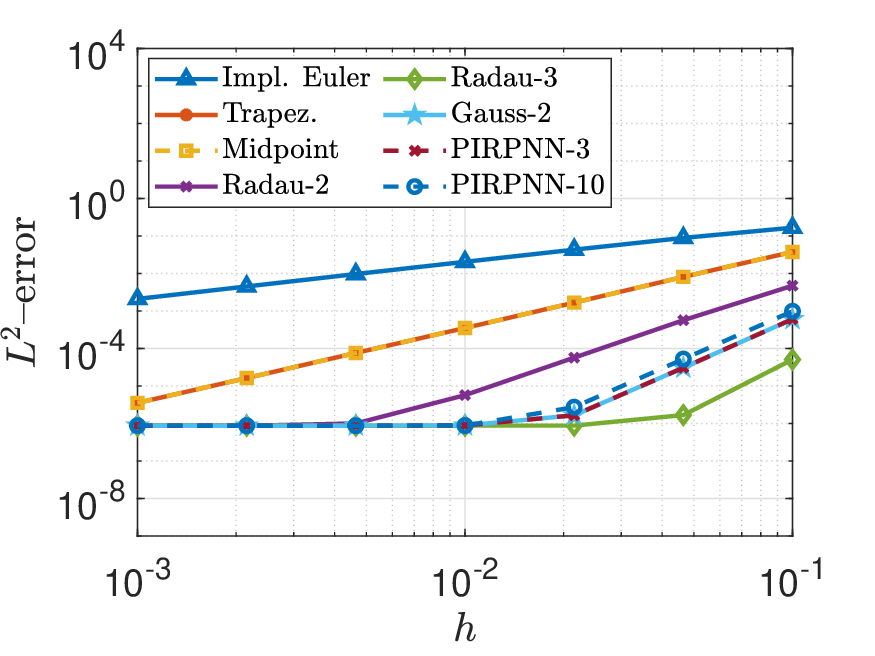}
    }
    \subfigure[]{ 
    \includegraphics[trim={0 0 0.5cm 0.5cm},clip,width=0.45\textwidth]{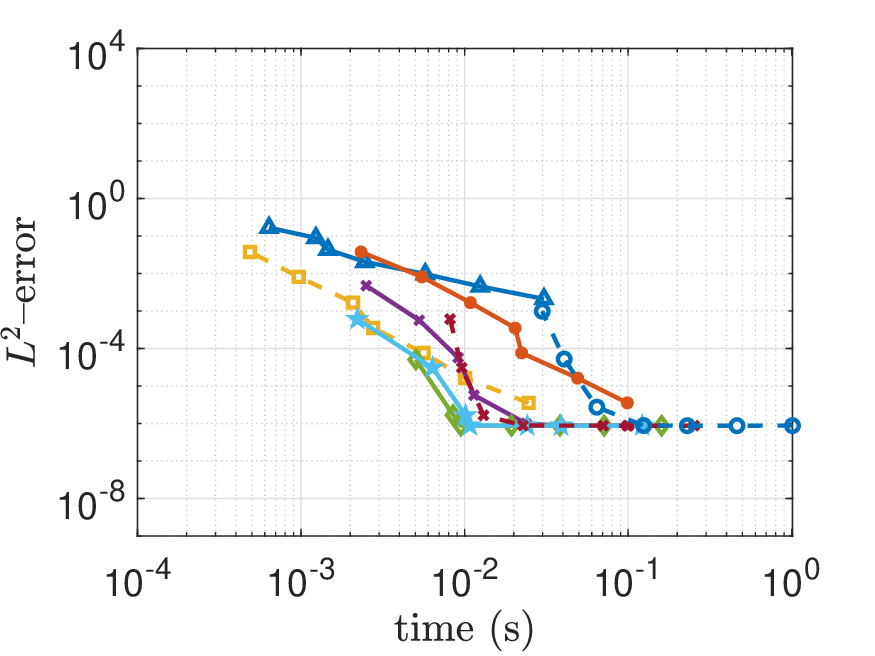}
    }
    \vspace{-2mm}
    \caption{The linear Diffusion-Reaction PDE in Eq \eqref{eq:pdelinear}, with $\nu=0.1$ and $\lambda=10$ in the time interval $[0,\, 1]$, discretized with FD with $n=100$. For the RPNNs we used $M=3$ and $M=10$ collocation points in each sub-interval of size $h$ and $N=9$ and $N=30$ neurons, respectively. The results with the PIRPNNs are reported in dashed lines, while implicit RK schemes with solid lines. (a) $L^2$--error in terms of fixed time step $h$. (b) $L^2$--error with respect to machine/computational execution time.}
\label{fig:PDE_reaction_diffusion}
\end{figure}
For the discretization in space of the PDE, we employed a second order centered finite difference scheme over a grid of $n+2$ points $x_i, i=0,1,\dots,n,n+1$. We select in particular $n=100$.
The boundary conditions are hardwired into the equations. The resulting system of $n$ ODEs is given:
\vspace{-1mm}
\begin{equation}\label{DIFFUSION-ODE}
 \frac{d u(x_i,t)}{dt}=\frac{d u_i}{dt}=\nu \frac{u_{i+1}-2u_{i}-u_{i-1}}{{\Delta x^2}}-\lambda u_i,\quad
    u_0=\frac{4u_1-u_2}{3 {\Delta x}}, \quad u_{n+1}=\frac{4u_n-u_{n-1}}{3 {\Delta x}},
\end{equation}
where ${\Delta x}=2\pi/(n+1)$.
The resulting ODE \eqref{DIFFUSION-ODE} is of the general form \eqref{eq:simplelinear0}, with the matrix $A$ corresponding to the finite difference approximation matrix for the 1-D Laplacian. The resulting system presents stiffness properties on account of the spectrum of the matrix $A$.
In this setting, the matrix $A$ representing the discretized PDE operator has the largest eigenvalue in absolute value equal to $-423.33$, indicating a highly stiff regime.
The numerical results, in the time interval $[0,1]$ using various fixed time steps $h$, of the proposed PIRPNNs, leveraging random projections, and the other traditional implicit schemes are depicted in Figure \ref{fig:PDE_reaction_diffusion}. 
%
%
The three-point and ten-point PIRPNN collocation schemes demonstrate the expected stability over the full range of time steps considered. While their accuracy is sub-optimal, due to fixed internal parameters not optimized for precision, they still surpass classical implicit schemes such as implicit Euler, implicit trapezoidal, midpoint, and Radau-2.
The three-points and ten-points PIRPNN collocation scheme also exhibit performance comparable to the Gauss-2 method.
However, PIRPNNs generally incur slightly higher computational costs due to the use of $N=3M$ neurons, as illustrated in panel (b) of Figure \ref{fig:PDE_reaction_diffusion}, this overhead becomes more evident when handling high-dimensional system (100 ODEs). 
Reducing RPNNs overparametrization, as suggested in \cite{fabiani2023parsimonious}, could enhance computational efficiency. Notably, the finite difference discretization limits the achievable accuracy for all methods, with a saturation around 1E$-$8 for $n=100$. 

\section{Conclusions}
{We presented a (constructive) linear stability analysis proof for the solution of (stiff) ODEs with physics-informed RPNNs (PI-RPNNs) using Gaussian radial basis functions as activation functions}. We demonstrated that PI-RPNNs, with appropriate sampling of the hyperparameters of the activation functions, are consistent and  unconditionally asymptotically stable schemes for (stiff) linear ODEs for all step sizes and that they have the correct stability.
This is the first time such constructive proof is given for the stability of physics-informed RPNNs for ODEs. \color{black}{We also demonstrated
that RPNNs are uniform approximators to the solution of ODEs.

Despite the fact that the primary focus of the current work is on the theoretical investigation of the stability properties of PI-RPNNs (done here for the first time), numerical comparisons on two benchmark stiff problems were also carried out indicatively, to evaluate the computational cost, convergence, and numerical approximation accuracy of the proposed PIRPNNs for various step sizes against traditional implicit schemes. These comparisons demonstrated that the proposed physics-informed RPNN achieves comparable 
numerical approximation accuracy to traditional implicit schemes while also maintaining a comparable computational cost across a wide range of step sizes. 
Extensive comparisons of such RPNNs schemes with other traditional schemes can be found in our earlier works \cite{fabiani2021numerical, fabiani2023parsimonious} where we showed, that optimally-designed PI-RPNNs rival traditional numerical  methods such as FD, FEMs, and adaptive stiff solvers for ODEs, while outperform physics-informed deep neural networks-based methods) by several orders of magnitude, both in approximation accuracy and computational cost.
\color{black}

We believe that our work will open the path for a more rigorous numerical analysis of scientific machine learning algorithms for the solution of both the forward and inverse problems for differential equations. 

\section*{Acknowledgements}
E.B. acknowledges support by ONR, ARO, DARPA RSDN, NIH, NSF, CRCNS.
C.S. acknowledges partial support from PNRR MUR Italy, projects PE0000013-Future Artificial Intelligence Research-FAIR \& CN0000013 CN HPC - National Centre for HPC, Big Data and Quantum Computing, and the Istituto di Scienze e Tecnologie per l'Energia e la Mobilità Sostenibili-CNR. C.S. \& E.B. acknowledge support from  Gruppo Nazionale Calcolo Scientifico-Istituto Nazionale di Alta Matematica (GNCS-INdAM). 
\bibliographystyle{plain}
\bibliography{references} 

\end{document}